\documentclass[10pt,reqno]{amsart}

\usepackage{amsmath,amscd}

\usepackage{tikz}
\usepackage{amsmath,amscd,amsfonts,amsthm,amssymb,latexsym,mathrsfs,url}

\theoremstyle{plain}
   \newtheorem{theorem}{Theorem}[section]
   
   \newtheorem{lemma}[theorem]{Lemma}
   \newtheorem{corollary}[theorem]{Corollary}

\theoremstyle{definition}
   
   \newtheorem{question}{Question}

\theoremstyle{remark}

\newcommand{\RR}{\mathbb{R}}

\def\newop#1{\expandafter\def\csname #1\endcsname{\mathop{\rm
#1}\nolimits}}

\def\multiset#1#2{\ensuremath{\left(\kern-.3em\left(\genfrac{}{}{0pt}{}{#1}{#2}\right)\kern-.3em\right)}}

\makeatletter
\newcommand{\rmnum}[1]{\romannumeral #1}
\newcommand{\Rmnum}[1]{\expandafter\@slowromancap\romannumeral #1@}
\makeatother

\def\newop#1{\expandafter\def\csname #1\endcsname{\mathop{\rm
#1}\nolimits}}

\newop{diag}
\newop{supp}
\newop{JP}
\newop{Sym}
\newop{si}
\newop{glb}
\newop{des}
\newop{fmaj}
\newop{comaj}
\newop{maj}
\newop{asc}
\newop{diag}
\newop{AB}
\newop{DB}
\newop{B}
\newop{AT}
\newop{DT}
\newop{peak}
\newop{rank}
\newop{des}
\newop{Sym}
\newop{exc}
\newop{sign}
\newop{wcr}
\newop{cro}
\newop{pc}
\newop{zc}
\newop{nc}
\newop{yy}

\title{Compatible polynomials and edgewise subdivisions}
\author{Madeleine Leander}
\address{Department of Mathematics, Stockholm University, SE-106 91
  Stockholm, Sweden}
\email{madde@math.su.se}

\begin{document}
\begin{abstract}
In this paper we answer a question posed by C. A. Athanasiadis. Namely, we prove that the local $h$-polynomial of the $r$th edgewise subdivision of the $(n-1)$-dimensional simplex $2^V$ has only real zeros. In doing so we find a tool, using compatible polynomials, that can be used to check if polynomials in certain types of families of polynomials have only real zeros. 

Corollary 2.5 of this paper was recently independently proved by P. B. Zhang \cite{Zhang}.
\end{abstract}
\maketitle

\thispagestyle{empty}

\section{Introduction}
\let\thefootnote\relax\footnote{This paper is part of my doctoral thesis \cite{ML}. }
Several polynomials of combinatorial nature are known to have only real zeros and the property of being real-rooted has lately been studied frequently within combinatorics.  
 
An important tool that has been used to prove that polynomials in certain families of polynomials are real-rooted is the technique of compatible polynomials. In \cite{ChudnovskySeymour} it was used to prove that all zeros of the independence polynomial of a claw-free graph are real. In \cite{SavageVisontai} it was used to prove that the $s$-Eulerian polynomial has only real zeros and that the type $D$ Eulerian polynomials are real-rooted. We will use this technique to prove that the local $h$-polynomial of the $r$th edgewise subdivision of the $(n-1)$-dimensional simplex $2^V$ has only real zeros. We then provide a tool that can be used to check whether polynomials in certain types of families of polynomials have only real zeros.  

Let $\Delta$ be an abstract simplicial complex of dimension $d-1$ with $f_i$ faces of dimension $i$. The $f$-vector of $\Delta$ is denoted by
$$f(\Delta) =(f_0, \ldots , f_{d-1}),$$
where $f_{-1} =1$. When studying subdivisions of $\Delta$ it is often more convenient to consider the $h$-vector of $\Delta$. The $f$-vector and the $h$-vector, $h(\Delta)=(h_0, \ldots h_d),$ uniquely determine each other through the linear relation
$$\sum_{i=0}^d f_{i-1}(x-1)^{d-i} = \sum_{i=0}^d h_i x^{d-i}.$$ 
As a further refinement, the local $h$-vector was introduced by Stanley in \cite{Stanley} as a tool to study subdivisions of simplicial complexes. 

Let $V$ be an $n$-element set. We denote by $2^V$ the simplex with vertex set $V$. 
For a positive integer $r$, taking the $r$th edgewise subdivision of a simplicial complex, $\Delta$, is a way to subdivide $\Delta$ such that every face $F\in \Delta$ is subdivided into $r^{\text{dim}(F)}$ faces of the same dimension, 
see \cite{Athanasiadis} for a further discussion. In the same paper the following question was addressed. It will be answered, to the affirmative, in Section~\ref{quest}.

\begin{question}[Athanasiadis] 
\label{QN}
Is the local $h$-polynomial of the $r$th edgewise subdivision of the $(n-1)$-dimensional simplex $2^V$, $l _V((2^V)^{\langle r \rangle},x)$, real-rooted?
\end{question}

The $n$th Eulerian polynomial may be defined as the generating polynomial for the ascent statistic over the symmetric group. 
The $r$th edgewise subdivision of the $(n-1)$-dimensional simplex $2^V$ can be expressed as the generating polynomial for the ascent statistic over a restricted version of the Smirnov words, see \cite{Athanasiadis}. Due to this result we can use some already developed techniques to show that these polynomials are indeed real-rooted. In Section \ref{gene} we generalize this method further. 

\section{Compatible polynomials and edgewise subdivisions}
\label{quest}
To  prove that the local $h$-polynomial of the $r$th edgewise subdivision of the $(n-1)$-dimensional simplex $2^V,$ denoted $l _V((2^V)^{\langle r \rangle},x),$ has only real zeros, we need combinatorial interpretations of these polynomials. 

Let $\mathcal{SW}(n,r)$ be the set of words $(w_0,w_1,\ldots , w_n) \in \{0,1, \dots r-1\}^{n+1}$ such that $w_i \ne  w_{i+1}$ for $i=0,1, \ldots , n-1,$ and $w_0=w_n=0$. The words in $\mathcal{SW}(n,r)$ are the Smirnov words with restrictions on the first and last letters, see \cite{LSW}. 
An \emph{ascent} is an index $i\in \{0, 1, \ldots , n-1\}$ such that $w_i<w_{i+1}.$
Let $\asc(w)$ denote the number of ascents in $w$.  
The polynomial $l _V((2^V)^{\langle r \rangle},x)$ can be expressed as 

$$l _V((2^V)^{\langle r \rangle},x) = \sum_{w \in \mathcal{SW}(n,r)}x^{\asc(w)},$$
see \cite{Athanasiadis}.

In order to show that $l _V((2^V)^{\langle r \rangle},x)$ has only real zeros we consider the following polynomials:
\begin{eqnarray}\label{refinement}
E_{r,n}^i = \sum_{w \in \mathcal{SW'}(n,r)}  \chi (w_n=i ) x^{\asc(w)},
\end{eqnarray}
where $\chi (\varphi )$ is one if  $\varphi$ is true and zero if not, and where $\mathcal{SW'}(n,r)$ is the set of words $(w_0,w_1,\ldots , w_n) \in \{0,1, \dots r-1\}^{n+1}$ such that $w_i \ne w_{i+1}$ for $i=0,1, \ldots , n-1$ and $w_0=0$. That is, we do not require $w_n$ to be zero. 
Note that $E_{r,n}^0$ is the polynomial of interest for this section as $E_{r,n}^0=l _V((2^V)^{\langle r \rangle},x)$.

\begin{lemma}
\label{reclemma}
Let $i,r,n$ be nonnegative integers with $n\geq 1, r \geq 2$ and $ i \leq r-1$. Then 
\begin{eqnarray}\label{recurrence}
E_{r,n}^i= \sum_{h=0}^{i-1} x E_{r,n-1}^h + \sum_{h=i+1}^{r-1}  E_{r,n-1}^h .
\end{eqnarray}
\end{lemma}

\begin{proof}
First we split $E_{r,n}^i$ into two parts depending on if $w_{n-1}$ is less than or greater than $i$:
\begin{eqnarray*}
E_{r,n}^i&=&\sum_{w \in \mathcal{SW'}(n,r)}  \chi (w_n=i ) x^{\asc(w)} \\
&=&\sum_{\substack{w \in \mathcal{SW'}(n,r) \\ w_{n-1}<i}}  \chi (w_n=i ) x^{\asc(w)} + 
\sum_{\substack{w \in \mathcal{SW'}(n,r) \\ w_{n-1}>i}}  \chi (w_n=i ) x^{\asc(w)}.
\end{eqnarray*}
Let us denote the first sum of the last line by $S_1$ and the second by $S_2$.
If $w_{n-1}=h<i$ we have an ascent at $n-1$. If we remove $w_n$ the resulting word ends with $h$ and is of size one less. 
By doing so we also remove the ascent $n-1$. Hence $\sum_{h=0}^{i-1} x E_{r,n-1}^h =S_1$.  
Now, if $w_{n-1}=h>i$ and we remove $w_n$ from the word $w$, no ascent is affected and thus $S_2=\sum_{h=i+1}^{r-1}  E_{r,n-1}^h.$
\end{proof}

To prove Question \ref{QN} we will study compatibility of the polynomials $E_{r,n}^i$. 
Let $f_1(x), \ldots , f_k(x) $ be polynomials with real coefficients. Then they are said to be \emph{compatible} if, for all real and positive $c_1,\ldots , c_k,$ the polynomial 
$$ \sum_{j=1}^k c_jf_j(x)
$$
has only real zeros. The polynomials are said to be pairwise compatible if, for all $i,j\in \{1,\ldots ,k\}$, $f_i(x)$ and $f_j(x)$ are compatible. Chudnovsky-Seymour proved the following.

\begin{lemma}[Chudnovsky-Seymour \cite{ChudnovskySeymour}, 2.2]
\label{pairwise}
If the polynomials $f_1(x), \ldots , f_n(x)$ have positive leading coefficients, then they are pairwise compatible if and only if they are compatible.
\end{lemma}

\begin{theorem}
\label{comp}
Let $f_1, \ldots ,f_n \in \mathbb{R}[x]$ be a sequence of polynomials with positive leading coefficients such that 
\begin{itemize}
\item[(a)] $f_i(x)$ and $f_j(x)$ are compatible, and
\item[(b)] $xf_i(x)$ and $f_j(x)$ are compatible
\end{itemize}
for all $1\leq i \leq j \leq n.$

Define a set of polynomials $g_1, \ldots , g_n \in \mathbb{R}[x] $ by 
$$
g_k(x) = \sum_{h=1}^{i-1} x f_h(x) + \sum_{h=i+1}^n f_h(x). 
$$
Then 
\begin{itemize}
\item[(a')] $g_i(x)$ and $g_j(x)$ are compatible, and
\item[(b')] $xg_i(x)$ and $g_j(x)$ are compatible
\end{itemize}
for all $1\leq i \leq j \leq n.$
\end{theorem}

The theorem above is proved in a similar manner as that of Theorem 2.3 in \cite{SavageVisontai}.

\begin{proof}
First we show (a'), that $c_ig_i(x) + c_jg_j(x)$ has only real zeros for all $c_i,c_j \geq 0.$ By the definition of $g_i$ and $g_j$ we have
\begin{eqnarray*}
c_ig_i(x) + c_jg_j(x) &=&  \sum_{\alpha=1}^{i-1} (c_i+c_j)x f_{\alpha}(x) + c_jxf_i + \sum_{\beta=i+1}^{j-1} (c_i + c_jx)f_{\beta}(x) \\
&+& c_if_j + \sum_{\gamma=j+1}^{n} (c_i + c_j)f_{\gamma}(x).
\end{eqnarray*}
Now, by Lemma \ref{pairwise} it suffices to prove that any two polynomials from any of the three sets
\begin{itemize}
\item[$A=$] $\{xf_{\alpha}(x) : 0\leq \alpha \leq i\},$
\item[$B=$] $ \{ (c_i + c_jx)f_{\beta}(x) : i<\beta < j \}$, and
\item[$C=$] $ \{ f_{\gamma}(x) : j \leq \gamma \leq n \}$,
\end{itemize}
are pairwise compatible. 
By (a) two polynomials in $A$ are compatible. The same is true for two polynomials in $B$ or $C$. By (b) any polynomial in $A$ is compatible with any polynomial in $C$.  
To prove that a polynomial in $A$ is compatible with one in $B$ we will consider the two sets $B_1=\{ c_i f_{\beta}(x) \}$ and $B_2=\{c_jx f_{\beta}(x) \}.$

Now, any two polynomials from $A,B_1$ and $B_2$ are pairwise compatible by (a) and (b), and by Lemma \ref{pairwise} any three polynomials from $A,B_1$ and $B_2$ are compatible. Hence polynomials in $A$ and $B$ are compatible since conic combinations of three polynomials, one from $A,B_1$ and $B_2$, are real-rooted. To prove that a polynomial in $B$ is compatible with one in $C$ we use an analogous argument utilizing the fact that polynomials $ f_{\beta}(x), xf_{\beta}(x)$ and $ f_{\gamma}(x)$ are pairwise compatible.

It remains to show (b'), that $c_ixg_i(x) + c_jg_j(x)$ has only real zeros for all $c_i,c_j \geq 0.$ It is done similarly to the proof of (a'). First, by definition of $g_i$ and $g_j$ we have
\begin{eqnarray*}
c_ixg_i(x) + c_jg_j(x) &=&  \sum_{\alpha=1}^{i-1} (c_ix+c_j)x f_{\alpha}(x) + c_jxf_i + \sum_{\beta=i+1}^{j-1} (c_i + c_j)xf_{\beta}(x)\\
&+& c_ixf_j + \sum_{\gamma=j+1}^{n} (c_ix + c_j)f_{\gamma}(x).
\end{eqnarray*}
By Lemma \ref{pairwise} it suffices to prove that the polynomials in the three sets 
\begin{itemize}
\item[$A'=$] $\{x(c_ix+c_j)f_{\alpha}(x) : 0\leq \alpha < i\},$
\item[$B'=$] $ \{xf_{\beta}(x) : i\leq \beta \leq j\}$,  and
\item[$C'=$] $\{(c_ix+c_j)f_{\gamma}(x) : j < \gamma \leq n,$
\end{itemize}
are pairwise compatible. 
Again, two polynomials from the same set $A',B'$ or $C'$ are compatible by (a). Polynomials in $A'$ and $B'$ are compatible since $x^2f_{\alpha}(x), xf_{\alpha}(x)$ and $x^2f_i(x)$ are pairwise compatible by (a) and (b). 
Polynomials from $A'$ and $C'$ are compatible by (b). Finally $xf_{\beta}(x),  xf_{\gamma}(x)$ and $ f_{\gamma}(x)$ are compatible by (a) and (b) which shows compatibility for polynomials in $B'$ and $C'$.
\end{proof}

We are now ready to  answer to Question \ref{QN}. First we prove real-rootedness for the polynomials $E_{r,n}^i$.

\begin{theorem}
For nonnegative integers $i,r,n$  with $n\geq 1, r \geq 2$ and $ 0 \leq i \leq r-1$ the polynomials $E_{r,n}^i$   
has only real zeros. 
\end{theorem}

\begin{proof}
We use induction on $n$ and prove the stronger statement, that for each $n$ and $r$, the sequence $\{E_{r,n}^i\}_{i=0}^{r-1}$ satisfies (a) and (b) in Theorem \ref{comp}. Clearly the statement holds for the base case $n=1.$ Now assume $n \geq 2$.
By Lemma \ref{reclemma} we have the recurrence
\begin{eqnarray}
E_{r,n}^i= \sum_{h=0}^{i-1} x E_{r,n-1}^h + \sum_{h=i+1}^{r-1}  E_{r,n-1}^h 
\end{eqnarray}
and thus the induction step holds by Theorem \ref{comp}.
In particular $E_{r,n}î$ has only real zeros for $n\geq 1$. 
\end{proof}

For the special case of $i=0$ we get the following corollary and the answer to Question \ref{QN}. 

\begin{corollary}
The local $h$-polynomial of the $r$th edgewise subdivision of the $(n-1)$-dimensional simplex $2^V$, $l _V((2^V)^{\langle r \rangle},x),$ has only real zeros. 
\end{corollary}

\section{A generalization}
\label{gene}

Let $f$ and $g$ be two real-rooted polynomials in $\RR[x]$ with positive leading coefficients. Let further $\alpha_1 \geq \cdots \geq \alpha_n$ and $\beta_1 \geq \cdots \geq \beta_m$ be the zeros of $f$ and $g$, respectively. If 
$$\cdots \leq \alpha_2 \leq \beta_2 \leq \alpha_1 \leq \beta_1$$
we say that $f$ is an \emph{interleaver} of $g$ and we write $f \ll g$. We also let $f\ll 0$ and $0\ll f$.

We call a sequence $F_n=(f_i)^n_{i=1}$ of real-rooted polynomials \emph{interlacing} if $f_i \ll f_j$ for all $1\leq i < j \leq n$. We denote by $\mathcal{F}_n $ the family of all interlacing sequences $(f_i)_{i=1}^n$ of polynomials and we let $\mathcal{F}^+_n$ be the family of $(f_i)_{i=1}^n \in \mathcal{F}_n $ such that $f_i$ has nonnegative coefficients for all $1\leq i \leq n$. 
We are interested in when an $m \times n$-matrix $G=(G_{i,j}(x))$ maps $ \mathcal{F}^+_n$ to $\mathcal{F}^+_m $ by $G\cdot F_n = (g_1,\ldots , g_m)^T$, where $g_k=\sum^n_{i=1} G_{ki}f_i.$ This problem was considered in \cite{HEC,F,SavageVisontai} and in \cite{HEC} the following theorem was proven.

\begin{theorem}[\cite{HEC}]
Let $G=(G_{i,j}(x))$ be an $m \times n$ matrix of polynomials. Then $G : \mathcal{F}^+_n \rightarrow \mathcal{F}^+_m $ if and only if 
\begin{itemize}
\item[1.] The coefficients of $G_{ij} $ are nonnegative for all $i\in [m] $ and $j \in [n]$ , and
\item[2.] for all $\lambda, \mu >0, 1 \leq i<j \leq n$ and $1\leq k<l\leq m$
\end{itemize}
\begin{eqnarray}
\label{ll}
(\lambda x + \mu)G_{kj}(x) + G_{lj}(x) \ll (\lambda x + \mu)G_{ki}(x) + G_{li}(x)
\end{eqnarray}
\end{theorem}

\begin{corollary}[\cite{HEC}]
\label{posdet}
Let $G=(G_{ij}) $ be an $m \times n$ matrix over the nonnegative real numbers. Then $G : \mathcal{F}^+_n \rightarrow \mathcal{F}^+_m $ if and only if the determinants of all $2\times 2$ sub-matrices are nonnegative.
\end{corollary}

What was discussed in Section 2, in particular Theorem \ref{comp}, can be thought of as matrices $G$ where the entries are $0$ on the main diagonal, $1$ above and $x$ below the main diagonal. Similarly the matrices that were  considered in Theorem 2.3 in \cite{SavageVisontai} are those with 1 on and above the main diagonal and $x$ below. We will generalize the result to other matrices $G$. The goal of this section is to determine for which matrices $G$ with entries in $\{0,1,x\}$ all $2 \times 2$ submatrices of $G$,
\[ \left( \begin{array}{cc}
 G_{ki}&  G_{kj} \\
 G_{li}&  G_{lj}  \end{array} \right), \]
 satisfy (\ref{ll}) and thus are such that $G$ preserves interlacing.
 
 The following theorem tells us what must be avoided.

\begin{lemma}
\label{avoid}
If an $m\times n$ matrix $G=(G_{i,j}(x))$ has a $2 \times 2$ sub-matrix of the form
  \begin{itemize}
  \item[\rmnum{1})] $ \left( \begin{array}{cc}
x & G_{kj}\\
1 & G_{lj} \end{array} \right)$ or 
 $\left( \begin{array}{cc}
G_{ki}& x  \\
 G_{li} &  1  \end{array} \right)$,

\item[\rmnum{2})] $\left( \begin{array}{cc}
1&  x\\
G_{li} & G_{lj}   
 \end{array} \right)$ 
or
 $\left( \begin{array}{cc}
G_{ki}& G_{kj}  \\
 1 &  x  \end{array} \right)$,
 
   \item[\rmnum{3})] $\left( \begin{array}{cc}
1& G_{kj}  \\
  G_{li}&  x  \end{array} \right)$, 
 $\left( \begin{array}{cc}
x& G_{kj}  \\
 G_{li} &  1  \end{array} \right)$ 
or
 $\left( \begin{array}{cc}
G_{ki}& x  \\
 1 &  G_{lj}  \end{array} \right)$

except for 
 $\left( \begin{array}{cc}
1& 1 \\
 x&  x  \end{array} \right)$ and
 $\left( \begin{array}{cc}
x&  1\\
x &  1 \end{array} \right)$,
  \item[\rmnum{4})] all entries are $0$ or $1$ and the determinant is negative, or

 \item[\rmnum{5})] $x$ times a zero-one matrix with negative determinant,

 \end{itemize}
then $G : \mathcal{F}^+_n \nrightarrow \mathcal{F}^+_m $.

 \end{lemma}

 \begin{proof}
 A $2 \times 2$ matrix as in \rmnum{1}) must be avoided since for $\lambda = \mu =1$ one side of (\ref{ll}) will not be real-rooted. 
 For ii) we have two cases. If the $(1\,\, x)$ is in the first row and we let $\lambda \rightarrow \infty$ in (\ref{ll}) we get $x^2 \ll x$ which is not true. 
 If the second row is $(1\,\, x)$ and we let $\lambda, \mu \rightarrow 0$ we see that (\ref{ll}) gives $x \ll 1$, another false statement. 
For iii) there are seven matrices to check, that are not included in i) or ii). We check them by hand. For example if $G_{kj}=G_{li}=0, G_{ki}=x $ and $G_{lj}=1$ we get $1 \ll (\lambda x + \mu)x$ from (\ref{ll}), which is not true for all $\lambda$ and $\mu$. 
The last two cases follows from Corollary \ref{posdet}.
 \end{proof}

 \begin{theorem}
 If an $m\times n$ matrix $G=(G_{i,j}(x))$ has no $2 \times 2$ submatrices of the type described in i)-v) in Lemma \ref{avoid} above and if all entries in $G$ are $0,1$ or $x$, then $G : \mathcal{F}^+_n \rightarrow \mathcal{F}^+_m $.
 \end{theorem}
 
 \begin{proof}
 One readily checks (\ref{ll}) for the seven matrices, 
 $$
   \left( \begin{array}{cc}
1& 0  \\
  0&  1  \end{array} \right),
 \left( \begin{array}{cc}
1& 0  \\
 x &  1  \end{array} \right),
 \left( \begin{array}{cc}
1& 1 \\
 0&  1  \end{array} \right),
 \left( \begin{array}{cc}
1&  1\\
x &  1 \end{array} \right),
 \left( \begin{array}{cc}
1& 0  \\
  1&  1  \end{array} \right),
 \left( \begin{array}{cc}
0& 0  \\
 1 &  0  \end{array} \right), 
 \left( \begin{array}{cc}
0& 1 \\
 x&  0  \end{array} \right).
  $$
  For example the second matrix corresponds to
  $$1 \ll \lambda x + \mu + x,$$
  which is true for all $\lambda,\mu\ \geq 0$.
  
By multiplying the matrices above, simple computer calculations show that we end up with 40 matrices with entries in $\{0,1,x\}$. They all satisfy (\ref{ll}) since if two matrices $A$ and $B$ satisfy (\ref{ll}) the same is true for $AB$ and $BA,$ since the property of preserving interlacing is closed under composition. Lemma \ref{avoid} counts for 41 $2 \times 2$ submatrices and thus they are the only $2 \times 2$ submatrices that needs to be avoided.
 \end{proof}

Note that the seven matrices presented in the proof above are generating all $2\times 2$ sub-matrices with entries in $\{0,1,x\}$ that satisfies (\ref{ll}). Thus we can use them to check if any $m \times n$ matrix $G=(G_{i,j}(x))$  is such that $     G : \mathcal{F}^+_n \rightarrow \mathcal{F}^+_m $.

 \begin{corollary}
\label{ferr}
Let $G=(G_{i,j}(x))$ be an $m\times n$ matrix with all entries in $\{0,1,x\}$ such that: 
\begin{itemize} 
\item[1)] If $G(i,j)=1$, then $G(k,l)=1$ whenever $1 \leq k \leq i$ and $j\leq l \leq n$. That is, all entries that are equal to 1 in $G$ form a Ferrers board in the top right corner of $G$. 
\item[2)]  If $G(i,j)=x$, then $G(k,l)=x$ whenever $i \leq k \leq m$ and $1\leq l \leq j$. That is, all entries that are equal to $x$ in $G$ form a Ferrers board in the bottom left corner of $G$.   
\end{itemize}
Then $G : \mathcal{F}^+_n \rightarrow \mathcal{F}^+_m $.
\end{corollary}

\begin{proof}
The possible $2 \times 2$ sub-matrices of $G$ are
\[ \left( \begin{array}{cc}
 1&  1 \\
 1&  1  \end{array} \right), 
\left( \begin{array}{cc}
1 &1   \\
0 & 1   \end{array} \right),
\left( \begin{array}{cc}
 1& 1  \\
x &  1  \end{array} \right),
 \left( \begin{array}{cc}
1 &  1 \\
x &  x  \end{array} \right),
 \left( \begin{array}{cc}
1 &  1 \\
x &  0  \end{array} \right),
 \left( \begin{array}{cc}
1 & 1  \\
0 &  0  \end{array} \right),
\left( \begin{array}{cc}
x & 1  \\
x &  1  \end{array} \right),
\]  

\[
 \left( \begin{array}{cc}
0 &  1 \\
x &  1  \end{array} \right),
 \left( \begin{array}{cc}
0 &  1 \\
0 &  1  \end{array} \right)
\left( \begin{array}{cc}
 x& 1  \\
x &  x  \end{array} \right),
 \left( \begin{array}{cc}
 0&  1 \\
 x& x   \end{array} \right),
 \left( \begin{array}{cc}
x &  1 \\
x &   0 \end{array} \right),
  \left( \begin{array}{cc}
 0&  1 \\
 x&   0 \end{array} \right),
 \left( \begin{array}{cc}
 0&  1 \\
 0&  0  \end{array} \right),
  \]
 
 \[
  \left( \begin{array}{cc}
0 &  0 \\
0 &   0 \end{array} \right),
 \left( \begin{array}{cc}
0 & 0  \\
x &  0  \end{array} \right),
 \left( \begin{array}{cc}
0 & 0  \\
x &   x \end{array} \right),
 \left( \begin{array}{cc}
 x&  0 \\
 x&  0  \end{array} \right)
,
\left( \begin{array}{cc}
x &   0\\
x &   x \end{array} \right)
\, \text{and} \,
\left( \begin{array}{cc}
x &  x \\
x &  x  \end{array} \right).
\]
The corollary can be proved either by noting that none of the matrices, $A$, above are of the form \rmnum{1})-\rmnum{5}) in Lemma \ref{avoid} or by noting that they all satisfy
$$
(\lambda x+ \mu ) A_{1,2} + A_{2,2} \ll  (\lambda x+ \mu ) A_{1,1} + A_{2,1},
$$
for all $\lambda,\mu > 0.$
\end{proof}
Examples of $\{0,1,x\}$-matrices satisfying 1) and 2) in Corollary \ref{ferr} are
 $$
   \left( \begin{array}{cccc}
1& 1&1&1  \\
0&1&1&1 \\
x& 1&1&1\\
x&x&x&x  \end{array} \right)
\text{ and }
   \left( \begin{array}{cccc}
0& 0&0&1  \\
x&0&0&0 \\
x& x&0&0\\
x&x&x&x  \end{array} \right).
$$
 
As an application of Corollary \ref{ferr} we now prove that generating polynomials of the ascent statistic over generalized Smirnov words are real-rooted. 
Let $\gamma=(\gamma_0, \ldots , \gamma_{r-1})$ be an integer sequence with $0 \leq \gamma_i \leq r-2$ for all $0\leq i \leq r-1$ such that $|\gamma_{i+1}-\gamma_i| \leq 1$ for all $0\leq i \leq r-2$. 
Let $\mathcal{SW}_{\gamma}(n,r)$ be the set of words $(w_0,w_1,\ldots , w_n) \in \{0,1, \dots r-1\}^{n+1}$ such that if $w_j =i$ then  
$|w_j -  w_{j-1}| > \gamma_i$ for $1\leq j \leq n$ and $w_0=w_n=0$. 
\begin{corollary} With $n,r$ and $\gamma$ as above, the polynomial
$$\sum_{w\in \mathcal{SW}_{\gamma}(n,r)}  x^{\asc(w)}$$
is real-rooted.
\end{corollary}
\begin{proof}
Let $\mathcal{SW'}_{\gamma}(n,r)$ be defined as $\mathcal{SW}_{\gamma}(n,r)$ without the restriction $w_n=0$. 
Let now
$$
E_{\gamma,r,n}^i = \sum_{w \in \mathcal{SW'}_{\gamma}(n,r)}  \chi (w_n=i ) x^{\asc(w)}.
$$
With the same reasoning as for the proof of Lemma 2.1 we may write
$$(E_{\gamma,r,n}^0, \ldots , E_{\gamma,r,n}^{r-1}) =A \cdot (E_{\gamma,r,n-1}^0, \ldots , E_{\gamma,r,n-1}^{r-1 })^t,$$
where $A=(a_{ij})_{i,j=0}^{r-1}$ is an $r \times r$ matrix with
$$
a_{ij}=
\begin{cases}
0 \text{ if } |i-j|  \leq \gamma_{i},\\
1 \text{ if } j -i > \gamma_{i}, \\
x \text{ if } i-j > \gamma_{i}.
\end{cases}
$$
Clearly $A$ satisfies 1) and 2) in Corollary \ref{ferr}. Thus $E_{\gamma,r,n}^i$ is real-rooted for all $0\leq i \leq r-1$.
\end{proof}

It would be interesting to further characterize classes of matrices for which all $2 \times 2$ submatrices satisfy (\ref{ll}). In particular, can one characterize the class of matrices generated by the $\{0,1,x\}$-matrices satisfying (\ref{ll})? Do we obtain all integer interlacing preserving matrices, or a proper subclass? 
\section*{Acknowledgements}
I want to thank Petter Br\"and\'en for fruitful discussions and valuable comments.


\begin{thebibliography}{20}

\bibitem{Athanasiadis}
C. A. Athanasiadis.
\newblock A survey of subdivisions and local $h$-vectors.
\newblock {\em Stanley 70th birthday volume, AMS (to appear)}, Preprint: {\em ArXiv e-prints, 1403.7144} (2014).

\bibitem{HEC}
P. ~Br\"and\'en.
\newblock Unimodality, Log-Concavity, Real-Rootedness and Beyond, 
\newblock {\em Handbook of Enumerative Combinatorics}, 
437--483, {\em Discrete Math. Appl.}, CRC Press, Boca Raton, FL (2015).

\bibitem{ChudnovskySeymour}
M. Chudnovsky, P. Seymour.
\newblock The roots of the independence polynomial of a claw-free graph.
\newblock {\em J. Combin. Theory Ser. B}, {\bf 97}(3), (2007), 350--357.

\bibitem{F}
S. Fisk.
\newblock Polynomials, roots, and interlacing. 
\newblock {\em ArXiv e-prints, 0612833} (2006).

\bibitem{Fro}
G.~Frobenius. 
\newblock \"Uber die Bernoulli'sehen zahlen und die Euler'schen polynome. 
\newblock {\em Sitzungsberichte der K\"oniglich Preussischen Akademie der Wissenschaften} (1910), zweiter Halbband.

\bibitem{ML}
M. Leander.
\newblock Combinatorics of stable polynomials and correlation inequalities (Doctoral dissertation). 
\newblock{\em Department of mathematics, Stockholm University} (2016-05-02).  ISBN: 978-91-7649-375-5. http://urn.kb.se/resolve?urn=urn:nbn:se:su:diva-128584.

\bibitem{LSW}
S. Linusson, J. Shareshian, M. L. Wachs.
\newblock Rees products and lexicographic shellability
\newblock{\em J. Combinatorics} {\bf 3} (2012), 243--276.

\bibitem{SavageVisontai}
C. D. Savage, M. Visontai.
\newblock The $s$-Eulerian polynomials have only real roots.
\newblock{\em Transactions of the American Mathematical Society} {\bf 367}(2), (2015), 1441--1466.

\bibitem{Stanley}
R. P. Stanley.
\newblock Subdivisions and local $h$-vectors.
\newblock {\em J. Amer. Math. Soc.} {\bf 5} (1992), 805--851.

\bibitem{Zhang}
P. B. Zhang.
\newblock On the Real-rootedness of the Local $h$-polynomials of Edgewise Subdivisions of Simplexes.
\newblock {\em ArXiv e-prints, 1605.02298} (2016).


\end{thebibliography}
\end{document}